\documentclass[10pt]{amsart}
     \makeatletter
     \def\section{\@startsection{section}{1}%
     \z@{.7\linespacing\@plus\linespacing}{.5\linespacing}%
     {\bfseries
     \centering
     }}
     \def\@secnumfont{\bfseries}
     \makeatother
\newtheorem{theorem}{Theorem}[section]
\newtheorem{lemma}[theorem]{Lemma}

\theoremstyle{definition}
\newtheorem{definition}[theorem]{Definition}
\newtheorem{example}[theorem]{Example}
\theoremstyle{remark}

\numberwithin{equation}{section}
\usepackage{amsmath,amsthm,amssymb,amsbsy}

\usepackage[all]{xy}
\usepackage{chngcntr}

\counterwithin{figure}{section}

\newcommand{\MM}{\mathbb{M}}

\newcommand{\RR}{\mathbb{R}}

\newcommand{\DDD}{\mathcal{D}}

\newcommand{\FFF}{\mathcal{F}}

\newcommand{\NNN}{\mathcal{N}}

\newcommand{\df}[1]{\,\mathrm{d}#1}                         
\newcommand{\dv}{\,\mathrm{d}}                         

\newcommand{\pa}{\mathrm{pa}}                                 

\begin{document}

\setlength{\parindent}{0cm}
\setlength{\parskip}{0.5cm}

\title[Causal interpretation of SDEs]{Causal interpretation of stochastic differential equations}

\author[A. Sokol and N. R. Hansen]{Alexander Sokol and Niels Richard Hansen}

\address{Alexander Sokol: Department of Mathematical Sciences, University of
  Copenhagen, Universitetsparken 5, 2100 Copenhagen \O, Denmark, alexander@math.ku.dk}

\address{Niels Richard Hansen: Department of Mathematical Sciences, University of
  Copenhagen, Universitetsparken 5, 2100 Copenhagen \O, Denmark, niels.r.hansen@math.ku.dk}



\begin{abstract}
We give a causal interpretation of stochastic differential equations (SDEs)
by defining the postintervention SDE resulting from an intervention in
an SDE. We show that under Lipschitz conditions, the solution to the
postintervention SDE is equal to a uniform limit in probability of
postintervention structural equation models based on the Euler scheme
of the original SDE, thus relating our definition to mainstream causal
concepts. We prove that when the driving noise in the SDE is a
L{\'e}vy process, the postintervention distribution is identifiable
from the generator of the SDE. 
\end{abstract}

\maketitle

\noindent

\section{Introduction}

The notion of causality has long been of interest to both statisticians and
scientists working in fields applying statistics.  In general, causal
models are models containing families of possible distributions of the
variables observed as well as appropriate mathematical descriptions of
causal structures in the data. Thus, claiming that a causal model is
true amounts to claiming more than statements about the
distribution of the variables observed. Causal modeling has several
goals, prominent among them are:

\begin{enumerate}
\item Estimation of intervention effects from fully or partially observed systems with a given causal structure.
\item Identification of the causal structure from observational data.
\end{enumerate}

One of the most developed theories of causal modeling is the approach
based on directed acyclic graphs (DAGs) and finitely many variables with no
explicit time component, descibed in \cite{MR1815675,MR2548166}. In recent
years, there have been efforts to develop similar notions of causality for
stochastic processes, both in discrete time and in continuous time. For
discrete-time results, see for example \cite{MR2354948,MR2925574,MR2575937}. As discrete-time models often are defined
through explicit functional relationships between variables, as in for
example autoregressive processes, such models fit directly into the
DAG-based framework. In the continuous-time framework, the uncountable
number of variables complicates the question of how to describe causal
relationships. 

Early discussions of causality in a continuous-time
framework can be found in \cite{MR916245,MR1403234,MR1395031}. One of
the most recent frameworks for causality in continuous time is based
on the concept of weak conditional local independence. For results related to this, see
\cite{MR2412641,MR2749916,MR2575938,MR2817610,KR2}. An
alternative notion of causality defined solely through filtrations is
developed in \cite{MR2596243,MR2811860}, and a notion of causality
in continuous time for ordinary differential equations is introduced in \cite{MooijJanzing_UAI_13}.

In Section 4.1 of \cite{MR2993496} it is noted that both ordinary
differential equations and stochastic differential equations (SDEs) allow for
a natural interpretation in terms of ``influence'', and that
interventions may be defined by substitutions in the differential
equations. In this paper, we make these ideas precise. Our main
contributions are:

\begin{enumerate}
\item For a given SDE, we give a precise definition of the
  postintervention SDE resulting from an intervention.
\item We show that under certain regularity assumptions, the
  solution of the postintervention SDE is the limit of a sequence of interventions in
  structural equation models based on the Euler scheme of the observational
  SDE.
\item We prove using (2) that for SDEs with a L{\'e}vy process as the driving
  semimartingale, the postintervention distribution is identifiable
  from the generator associated with the SDE.
\end{enumerate}

The definition (1) yields a generic notion of intervention effects for
SDEs applicable to causal inference in the case where an understanding
of the mechanisms of the system under consideration
is absent. The results of point (2) clarifies when we may expect this
generic notion to be applicable.

The result (3) is stated as Theorem
\ref{theorem:LevyDiffusionIdentifiability} and is the main theorem of
this paper. Its importance is as follows. In classical DAG-based models of causality such as developed
in \cite{MR2548166}, neither the DAG nor the effect of interventions
can be uniquely identified from the observational distribution. This
is one of the main difficulties of such causal models, and leads to a
rich and challenging theory for partially identifying intervention
effects, see for example \cite{MR2549555} and the references
therein. Theorem \ref{theorem:LevyDiffusionIdentifiability}
essentially shows that
for L{\'e}vy driven SDE models, the effect of interventions can be
uniquely identified from the observational distribution, meaning that
the intervention effect identification problem present in classical
DAG-based models vanishes for these SDE models.

We expect that this 
result will have considerable applicability for causal inference for
time-dependent observations. As argued in the series of examples comprised by Example
\ref{example:Yeast}, Example \ref{example:OUInterExample} and Example
\ref{example:YeastConclusion}, our results for example lead to a
dynamic modeling framework where  gene knockout effects can be derived from 
observational data -- a difficult problem which has previously only been dealt 
with, \cite{MaathuisNature,MR2549555}, using non-dynamic methods.

Of further particular note is that the identifiability result (3) in the list above
corresponds to a case where the error variables are not all independent, as is
otherwise often assumed to be the case when calculating intervention effects in the DAG-based framework. For
the DAG-based framwork, in the
case of independent errors, parts of the causal structure may be learned
from the observational distribution, as seen in \cite{VP}, and intervention distributions may be calculated by
a truncated factorization formula as in (3.10) of \cite{MR2548166}. For dependent
errors, such results are harder to come by. In our case, we
essentially take advantage of the Markov nature of the solutions to SDEs with L{\'e}vy
noise in order to obtain our identifiability result for SDE models, and we
are also able to obtain explicit descriptions of the resulting
postintervention distributions.

In matters of causality, it is important to distinguish clearly
between definitions, theorems and interpretations. Our definition of
postintervention SDEs will be a purely mathematical construct. It will,
however, have a natural interpretation in terms of causality. Given an SDE model, in order to
use the definition of postintervention SDEs given here to predict the effects
of real-world interventions, it is necessary that the SDE can be
sensibly interpreted as a data-generating mechanism with certain properties:
Specifically, as we will argue in Section
\ref{sec:InterventionInterpret}, it is essentially sufficient that the
driving semimartingales are autonomous in the sense that they may be
assumed not to be directly affected by interventions. This is an
assumption which is not testable from a statistical viewpoint.
It is, nonetheless, an assumption which may be justified by other means in
concrete cases.

The remainder of the paper is organized as follows. In Section \ref{sec:CausalSDE}, we
motivate and introduce our notion of intervention for SDEs. In Section \ref{sec:CausalStructures}, we
review the terminology of causal inference as developed in \cite{MR2548166}
and \cite{MR1815675}, based on structural equation models and directed
acyclic graphs. Section \ref{sec:InterventionInterpret} shows that under certain conditions,
our notion of intervention is equivalent to taking a limit of interventions in the context of structural equation
models based on the Euler scheme of the SDE. In Section
\ref{sec:Identifiability}, we give conditions for postintervention
distributions to be identifiable from the generator of the SDE. Finally, in Section
\ref{sec:Discussion}, we discuss our results. Appendix
\ref{sec:IdentifiabilityProof} contains proofs.

\section{Interventions for stochastic differential equations}

\label{sec:CausalSDE}

In this section, given an SDE, we define the notion of a
postintervention SDE, interpreted as the result of an intervention in a
system described by an SDE. This notion yields a causal interpretation
of stochastic differential equations.

We begin by considering three examples. Example \ref{example:Merton}
is a classical stochastic control problem. The control over a
stochastic process is achieved via a control variable, whose effect 
on the stochastic system is a part of the model assumptions. Such an 
assumption is an (implicit) assumption about a causal relationship, or
at least about how interventions in the system affect the system. Though 
the assumption is plausible in the specific example, we want to
bring attention to its existence. Example \ref{example:Yeast} discusses a case where our
stochastic model, due to the current state of knowledge in the subject
matter field, cannot be derived completely from background mechanisms of the system under
consideration. It is, however, highly desirable to be able to model
and discuss causality and the effect of interventions in this situation. Finally, Example
\ref{example:MotivatingSDE} provides an example where an understanding
of the background mechanisms of a system
provides an SDE model and also provides a candidate for how to
describe the effects of interventions in the system.

\begin{example}
\label{example:Merton}
Consider the following simplified variant of Merton's portfolio selection problem, first formulated in
\cite{Merton1969}. In this problem, we consider the Black-Scholes model for a
financial market in continuous time, consisting of a risk-free asset
with price process $B$ and a risky asset with price process $S$, following the SDEs
\begin{align}
  \dv B_t &= r B_t \dv t, \\
  \dv S_t &= \mu S_t \dv t + \sigma S_t \dv W_t.
\end{align}
Here, $r$ denotes the risk-free interest rate, $\mu$ is the expected
return of the risky asset, and $\sigma$ is the volatility of the risky
asset. Now consider an investor endowed with initial wealth $V$, who
invests a constant fraction $\alpha$ of his wealth at time $t$ in the
risky asset $S$ and holds the remaining fraction $1-\alpha$ of his wealth in the risk-free asset $B$.

Now, as the investor at time $t$ invests $(1-\alpha)V_t$ in the
risk-free asset, yielding ownership of $(1-\alpha)V_t/B_t$ units of this
asset, and invests $\alpha V_t$ in the risky asset, yielding
ownership of $\alpha V_t /S_t$ units of this asset, the arguments in
Chapter 6 of \cite{TB2009} yield that $V$ satisfies
\begin{align}
  \dv V_t &= (1-\alpha)(V_t/B_t) \dv B_t + \alpha (V_t/S_t) \dv
  S_t\notag\\
  &= (1-\alpha)V_t r \dv t + \alpha V_t \mu \dv t + \alpha V_t \sigma \dv W_t\notag\\
  &= ((r+\alpha(\mu-r))V_t)\dv t + \alpha V_t \sigma \dv W_t.
\end{align}
In \cite{Merton1969}, Merton endows the investor with a utility function
$u$, meaning that the utility for the investor of
having wealth $v$ is $u(v)$ and proceeds to solve the problem of
identifying the portfolio (how $\alpha$ should be dynamically chosen), which optimizes the lifetime value of the portfolio over $[0,T]$, given by
\begin{align}
  E e^{-rT}u(V_T),
\end{align}
subject to the constraint that $V_t>0$. The optimal (Markov) control
$\alpha(t, V_t)$, which is a function of time and wealth, can
generally be characterized as a solution to the Hamiton-Jacobi-Bellman
equation, and for some special choices of utility functions an
explicit analytic solution can be found. 

Now notice the following subtle point. In the above, we have
succesfully formulated an optimal control problem, seeking an optional
portfolio for the investor. At no point did it become necessary to
consider what the ``causal effect'' of a particular choice of
portfolio on the wealth process is, as the general financial arguments
of \cite{TB2009} provides for this: A change of
portfolio causes a change in the wealth process, while the opposite is
a somewhat insensible statement without a specified control process. 
This is an example of how, when we
have background knowledge of the effects of real-world choices (such
as the choice of portfolio) on terms of interest (the wealth of the
investor), the causal effects of choices, or interventions, are
determined by our background knowledge. In all these arguments there
is a hidden assumption, namely that the choice of portfolio doesn't
affect the Brownian motion that drives the price process. For small
investors this may be a reasonable assumptions, but it is well known
that large investors can affect the price process by their
investments. Thus in this classical control problem there are 
assumptions about how the control variable affects the system, and
this includes the assumption that the process driving the SDE 
is unaffected by the control variable -- a notion we later refer to
as autonomy of the driving process.
\hfill$\circ$
\end{example}

\begin{example}
\label{example:Yeast}
In this example we discuss the modeling of gene expression in 
the yeast microorganism \textit{Saccharomyces
  Cerevisiae}. The genome of this organism was the first eukaryotic
genome to be completely sequenced, see \cite{Gofiaeu1996}. In general, genes of an
organism are not active at all times, nor are they simply
active or not active. Instead, a gene has a level of expression,
indicating the production rate of the protein corresponding to the gene. An
important question in connection with genomic research is the
understanding of how the expression level of one gene influences the
expression level of other genes. An understanding of such causal networks
would allow analysis of what interventions to make on gene expression, for
example what genes to knock out (that is, turn permanently off) in
order to achieve some particular aim, such as optimal growth of an
organism or optimal production rate of a particular compound of interest.

For this particular microorganism, gene expression data are available,
both for non-mutated specimens and for mutations corresponding to
deletion of particular genes, see \cite{HughesEtAl2000}. Inference of the effect of interventions
based on gene expression levels of non-mutated specimens has
been carried out in \cite{MaathuisNature} using IDA (an acronym for ``Intervention calculus when the DAG is
absent''), see \cite{MR2549555}, and compared to intervention data resulting from
deletion mutants with favorable results.

The method investigated in \cite{MaathuisNature} is not based on a
dynamic model of gene expression, but rather on a multivariate Gaussian
model of cross sectional data. It suffers, for instance, from the
inability to include feedback loops. As a simple alternative suppose
that the $p = 5361$ genes of a non-mutant specimen of \textit{S. Cerevisiae}
evolves according to an 
Ornstein-Uhlenbeck process solving the SDE
\begin{align}
\label{eq:exampleOUForGenes}
  \dv X_t &= B(X_t - A)\dv t + \sigma \dv W_t,
\end{align}
where $B$ is a $p\times p$ matrix, $A$ is a $p$-dimensional vector,
$\sigma$ is a $p \times d$ matrix and $W$ is a $d$-dimensional
Brownian motion. One benefit of such a model is its mean reversion
properties, corresponding to gene expression levels fluctuating over
time, but generally remaining stable over periods of the life of the
specimen. Depending on the data available, standard statistical
methods may then be applied to obtain estimates of some or all of the parameters of
the model, yielding a description of the distribution of
our data.

As discussed above, the effect of knocking out gene $m$ (corresponding to setting $X^m$ to zero for
some $m$) in the model (\ref{eq:exampleOUForGenes}) is of central
importance. However, as we in this case do not have a sufficiently
detailed biochemical understanding of how genes influence each
other over time, it is less obvious than in Example
\ref{example:Merton} how the knockout intervention of gene $m$
affects the system. 

In other words, our lack of a generic concept for causality for SDEs,
applicable in the absence of knowledge of particular mechanisms of
causality, in this case prevents us from considering intervention
effects in our model. 
\hfill$\circ$
\end{example}

\begin{example}
\label{example:MotivatingSDE}
Chemical kinetics is concerned with the evolution of the concentrations
of chemicals over time, given in terms of a number of coupled chemical
reactions, see \cite{MR2222876}. In this example, we consider two
chemicals and derive a simple system of SDEs  from the fundamental
mechanisms of the chemical reactions. If the concentration of one
chemical is fixed (as an alternative to letting it evolve according
to the chemical reactions) the fundamental mechanisms allow us to
obtain an SDE for the concentrations of the remaining chemicals. This
SDE then describes the system after the intervention, and can be obtained from the
original system by a purely mechanical deletion and substitution
process. 

The chemicals are denoted $x$ and $y$ and the
corresponding concentrations are denoted $X$ and $Y$,
respectively. We assume that four reactions are possible, namely:
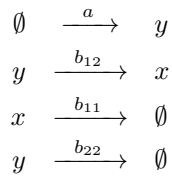
\begin{eqnarray*}
\emptyset & \xrightarrow{\hspace*{2mm} a \hspace*{2mm}} & y \\
y & \xrightarrow{\hspace*{2mm} b_{12} \hspace*{2mm}} & x \\
x & \xrightarrow{\hspace*{2mm} b_{11} \hspace*{2mm}} & \emptyset \\
y & \xrightarrow{\hspace*{2mm} b_{22} \hspace*{2mm}} & \emptyset
\end{eqnarray*}
Here, the first reaction denotes the creation or influx of chemical
$y$ with constant rate $a$, the second reaction denotes the change of
$y$ into $x$ at rate $b_{12} Y$, and the third and fourth reactions
denote degradation or outflux of $x$ and $y$ with
rates $b_{11}X$ and $b_{22}Y$, respectively. We collect the rates into 
the vector 
\begin{align}
\lambda(X, Y) = \left(\begin{array}{c}
a \\
b_{12} Y \\
b_{11} X \\
b_{22} Y \\
\end{array}
\right).
\end{align}
The so-called
stoichiometric matrix 
\begin{align}
S = \left( 
\begin{array}{cccc}
0 & 1 & -1 & 0 \\
1 & -1 & 0 & -1 
\end{array}
\right)
\end{align}
collects the information about the number of molecules, for each of
the two chemicals (rows), which are created or destroyed by each of the 
four reactions (columns). The rates $\lambda(X,Y)$ and the stoichiometric
matrix $S$ form the fundamental parameters of the system. We are interested
in using $\lambda(X,Y)$ and $S$ to construct a model for the evolution
of $X$ and $Y$ over time.

Several different stochastic and deterministic models
are available. One stochastic model is obtained by considering a Markov jump process on
$\mathbb{N}^2_0$, where each coordinate denotes the total number of
molecules of each chemical $x$ and $y$, and the transition rates are given
in terms of $S$ and $\lambda(X,Y)$. A system of SDEs approximating the
Markov jump process, see \cite{And}, is given by 
\begin{align}
\label{eq:chemSDE}
\left(
\begin{array}{c} 
X_t \\ 
Y_t
\end{array}
\right) 
= \left(
\begin{array}{c} 
X_0 \\ 
Y_0 + a t
\end{array}
\right) + \int_0^t
 B \left(
\begin{array}{c} 
X_s \\ 
Y_s
\end{array}
\right) \ \mathrm{d} s + 
\int_0^t \Sigma(X_s, Y_s) \ \mathrm{d} W_s
\end{align}
where $W_s$ denotes a four-dimensional Wiener process, and the
matrices $\Sigma(x,y)$ and $B$ are given by
\begin{align}
  \Sigma(x, y) &= S \text{diag} \sqrt{\lambda(x, y)} \notag\\
&= \left(\begin{array}{cccc}
0 & \sqrt{b_{12}y} & - \sqrt{b_{11}x} & 0 \\
\sqrt{a} & -\sqrt{b_{12}y} & 0 & -\sqrt{b_{22}y}
\end{array}\right)
\end{align}
and
\begin{align}
B = 
\left(\begin{array}{cc}
-b_{11} &  b_{12} \\
- b_{12} & - b_{22} 
\end{array}\right).
\end{align}
If we are able to fix the concentration $Y_t$ at a level $\zeta$, we effectively
remove the first and last of the reactions and the second will have
the constant rate $b_{12}\zeta$. By arguments as above we then derive the SDE 
\begin{align}
X_t &= X_0 + t b_{12} \zeta  -  \int b_{11} X_s \ \mathrm{d} s 
 +  \int_0^t \sigma(X_s) \ \mathrm{d} \widetilde{W}_s,
\end{align}
with $\widetilde{W}_s$ a two-dimensional Wiener process and $\sigma(x) = (\sqrt{b_{12} \zeta}, -\sqrt{b_{11} x})$.
We observe that this SDE, describing the intervened system, can be obtained from (\ref{eq:chemSDE})
by deleting the equation for $Y_t$ and substituting $\zeta$ for $Y_t$ in the remaining equation. 
\hfill$\circ$
\end{example}

We now proceed to our main definition. Recall that in Example
\ref{example:Yeast}, we were stopped short in our discussion of the
effect of interventions in our model due to the lack of a generic
notion of interventions for SDEs. We will now use the conclusions from Example
\ref{example:MotivatingSDE} to introduce such a generic notion of
interventions.

In the DAG-based framework, the DAG is a direct representation of the
causal structure of the system. We do not directly provide such a
representation of causality for SDEs. In general, the precise meaning of ``causation'' is a point of
contemporary debate, see for example \cite{MR1803167}. For our
purposes, it suffices to take a practical standpoint: The causal
structure of a system is sufficiently elucidated for the purposes of
our discussion if we know the
effects of making interventions in the system. For this reason, we restrict
ourselves in Definition \ref{def:SDEIntervention} to defining the effect of making interventions.

In Example \ref{example:MotivatingSDE}, we obtained results on the
effects of intervention in a system from a model for the entire system. In this
particular example, the resulting model for the intervention was justified by reference
to the fundamental mechanisms (the chemical reactions) driving the
system, and interventions resulted in SDEs modified by substitution and deletion. While noting that this
correspondence between interventions and substitution and deletion in the original
equations may not always be justified, we will use this principle as a
general, purely mathematical definition of interventions in SDEs.

Consider a filtered probability space
$(\Omega,\FFF,(\FFF_t)_{t\ge0},P)$ satisfying the usual conditions,
see \cite{MR2273672} for the definition of this and other notions related to
continuous-time stochastic processes. In order to formalize our
definition in a general framework, let $Z$ be a $d$-dimensional
semimartingale and assume that $a:\RR^p\to\MM(p,d)$ is a continuous
mapping, where $\MM(p,d)$ denotes the space of real $p\times d$
matrices. We consider the stochastic differential equation
\begin{align}
  X^i_t &= X^i_0 + \sum_{j=1}^d \int_0^t a_{ij}(X_{s-})\df{Z^j}_s,
  \qquad i\le p.\label{eq:MainSDE}
\end{align}
This SDE is written in integral form. Using differential and matrix
notation, (\ref{eq:MainSDE}) corresponds to the SDE $\dv X_t =
a(X_{t-})\dv Z_t$ with initial condition $X_0$. In the following,
$x^{-m}$ denotes the $(p-1)$-dimensional vector where the $m$'th
coordinate of $x \in \RR^p$ has been removed.

\begin{definition}
\label{def:SDEIntervention}
Consider some $m\le p$ and $\zeta : \RR^{p-1} \to \RR$. The stochastic differential
equation arising from (\ref{eq:MainSDE}) under the intervention
$X^m_t:= \zeta(X^{-m}_t)$ is the $(p-1)$-dimensional equation
\begin{align}
  (Y^{-m})^i_t &= X^i_0 + \sum_{j=1}^d \int_0^t b_{ij}(Y_{s-}^{-m})\df{Z^j}_s,
  \qquad i\neq m, \label{eq:IntervenedSDE}
\end{align}
where $b:\RR^{p-1}\to\MM(p-1,d)$ is defined by
$b_{ij}(y)=a_{ij}(y_1,\ldots,\zeta(y),\ldots,y_p)$ for $i\neq m$ and
$j\le d$ and the $\zeta(y)$ is on the $m$'th coordinate. 
\end{definition}

By Definition \ref{def:SDEIntervention}, intervening takes a $p$-dimensional SDE as
its argument and yields a $(p-1)$-dimensional SDE as its result. Note that existence and uniqueness
of solutions are not required for Definition \ref{def:SDEIntervention}
to make sense, although we will mainly take interest in cases where
both (\ref{eq:MainSDE}) and (\ref{eq:IntervenedSDE}) have unique
solutions. By Theorem V.7 of \cite{MR2273672}, this is for example the case whenever the
mappings $a$ and $\zeta$ are Lipschitz.

We stress that while Definition \ref{def:SDEIntervention} is motivated
by actual results from Example \ref{example:MotivatingSDE}, we do not
claim that it universally describes the effects of actual interventions
in a system. The discussion in Section \ref{sec:InterventionInterpret}
gives indications for \textit{whether} Definition
\ref{def:SDEIntervention} properly describes causality for a
particular SDE system. Our other results, such as those of Section
\ref{sec:Identifiability}, are devoted to analyze the consequences
\textit{if} Definition \ref{def:SDEIntervention} is a valid
description of the effect of interventions (and thus also a valid
description of the causal structure of the system, since knowing the
effects of interventions yields causal information about the system).

As discussed in Example \ref{example:Yeast}, an intervention with
a constant function $\zeta$ is of some interest, and in the context of gene
expression a knockout intervention, corresponding to $\zeta(y) = 0$, is one of the only
control mechanisms currently possible. If $\zeta$ is a constant we
identify the function with this constant, and we write $X_t^m :=\zeta$ for the intervention that puts the $m$'th coordinate constantly
equal to $\zeta$. 

Also note that the process $Y^{-m}$ above for which the SDE is formulated is a
$(p-1)$-dimensional process indexed by
$\{1,\ldots,p\}\setminus\{m\}$. When $Y^{-m}$ is a solution to
(\ref{eq:IntervenedSDE}), we also define $Y^m_t = \zeta(Y^{-m}_t)$,
and the $p$-dimensional process $Y$ is then the full result of making the
intervention $X^m:= \zeta(X^{-m}_t)$. The process $Y^{-m}$ is simply $Y$ with
its $m$'th coordinate removed. In general, the $p$-dimensional
process $Y$ will not satisfy any $p$-dimensional SDE except in special
cases. One such special case is when $\zeta$ is constant. In this case $Y$ will satisfy the
$p$-dimensional SDE
\begin{align}
  Y^i_t &= Y^i_0 + \sum_{j=1}^d \int_0^t c_{ij}(Y_{s-})\df{Z^j}_s,
  \qquad i\le p, \label{eq:pDimIntervenedSDE}
\end{align}
where $Y^i_0=X^i_0$ for $i\neq m$ and $Y^m_0=\zeta$, and $c:\RR^p\to\MM(p,d)$ is given by letting $c_{ij}(x)=a_{ij}(x)$ for
$i\neq m$ and $c_{mj}(x)=0$ for all $x\in\RR^p$ and $j\le d$.

Assuming that (\ref{eq:MainSDE}) and (\ref{eq:IntervenedSDE}) have unique
solutions for all interventions, we refer to (\ref{eq:MainSDE}) as the
observational SDE, to the solution of (\ref{eq:MainSDE}) as the observational
process, and to the distribution of the solution of
(\ref{eq:MainSDE}) as the observational distribution. We refer to
(\ref{eq:IntervenedSDE}) as the postintervention SDE, to the
solution of (\ref{eq:IntervenedSDE}) as the postintervention process
and to the distribution of the solution to
(\ref{eq:IntervenedSDE}) as the postintervention distribution.
Note how our definition of the postintervention
SDE has the same structure as the SDE obtained in Example
\ref{example:MotivatingSDE} by reference to fundamental mechanisms. 

As a first application of Definition \ref{def:SDEIntervention}, we
show in Example \ref{example:OUInterExample} that by Definition
\ref{def:SDEIntervention}, intervention with constant functions in an Ornstein-Uhlenbeck
process yields another Ornstein-Uhlenbeck process. Recalling Example
\ref{example:Yeast}, we thus find that if Definition
\ref{def:SDEIntervention} is applicable in the SDE model of Example
\ref{example:Yeast}, and if we can identify the correct parameters of
the SDE, then we can reason about the effects of interventions.

\begin{example}
\label{example:OUInterExample}
Let $x_0\in\RR^p$, $A\in\RR^p$, $B\in\MM(p,p)$ and
$\sigma\in\MM(p,d)$. The Ornstein-Uhlenbeck SDE with initial value
$X_0$, mean reversion level $A$, mean reversion speed $B$, diffusion
matrix $\sigma$ and $d$-dimensional driving noise is
\begin{align}
  X_t &= X_0 + \int_0^t B(X_s-A)\df{s}+\sigma W_t,\label{eq:OUSDE}
\end{align}
where $W$ is a $d$-dimensional $(\FFF_t)$ Brownian motion, see Section
II.72 of \cite{MR1796539}. Fix $m\le p$ and $\zeta\in\RR$. Under the intervention
$X^m:=\zeta$, we obtain that the postintervention process satisfies
\begin{align}
  Y^i_t &= X^i_0+\int_0^t \sum_{j\neq m}^p B_{ij}(Y^j_s-A_j)+B_{im}(\zeta-A_m)\df{s}+\sum_{j=1}^d\sigma_{ij} W^j_t.\label{eq:OUSDEInter}
\end{align}
for $i\neq m$. Now let $\tilde{B}$ be the submatrix of $B$ obtained by
removing the $m$'th row and column of $B$, and assume that $\tilde{B}$
is invertible. We then obtain
\begin{align}
  Y^{-m}_t &= X_0^{-m} + \int_0^t \tilde{B}(Y^{-m}_s-\tilde{A})\df{s}+\tilde{\sigma} W_t,\label{eq:OUSDEInter2}
\end{align}
where $\tilde{\sigma}$ is obtained by removing the $m$'th row of
$\sigma$ and $\tilde{A}=\alpha-\tilde{B}^{-1}\beta$, where $\alpha$
and $\beta$ are obtained by removing the $m$'th coordinate from $A$
and from the vector whose $i$'th component is $B_{im}(\zeta-A_m)$,
respectively. Thus, $Y^{-m}$ solves an $(p-1)$-dimensional Ornstein-Uhlenbeck SDE with
initial value $X_0^{-m}$, mean reversion level $\tilde{A}$, mean reversion
speed $\tilde{B}$ and diffusion matrix $\tilde{\sigma}$.
\hfill$\circ$
\end{example}

Now note that for the SDE 
\begin{align}
  \dv X_t &= B(X_t - A)\dv t + \sigma \dv W_t,
\end{align}
considered in Example \ref{example:Yeast}, the solution distribution depends only on $\sigma$ through
$\sigma\sigma^t$. Therefore, the parameters of the SDE are not
uniquely identifiable from the observational distribution. As we thus
cannot identify the parameters of the SDE, it appears that we
cannot identify the postintervention SDE in Definition
\ref{def:SDEIntervention}. 
In Example \ref{example:YeastConclusion}, we show how to use our main
theorem, Theorem \ref{theorem:LevyDiffusionIdentifiability}, on
identifiability of postintervention distributions to circumvent this
problem. Though we cannot identify the postintervention SDE, we can in fact
identify the postintervention distribution. 

Note also that in Example \ref{example:MotivatingSDE}, the matrix 
\begin{align}
\Sigma(X, Y) \Sigma(X, Y)^t = 
\left( \begin{array}{cc}
b_{12} Y + b_{11} X & - b_{12}Y \\
- b_{12} Y & a + b_{12} Y + b_{22} Y 
\end{array}\right)
\end{align}
is not diagonal, implying that the martingale parts of the
semimartingale $(X,Y)$ are not orthogonal. This shows that there are
naturally occuring situations where it is necessary to consider models
with non-orthogonal martingale parts. This is a situation excluded in
the WCLI framework of \cite{MR2575938} and is a motivating factor for
the level of generality in our definition.

\section{Terminology of SEMs, DAGs and interventions}

\label{sec:CausalStructures}

In this section, we review the basic notions related to intervention
calculus for structural equation models (SEMs). For a detailed overview, see
\cite{MR2548166,MR1815675}. We will use these notions in Section
\ref{sec:InterventionInterpret} to interpret our definition of
intervention for SDEs in terms of intervention calculus for structural
equation models.

As remarked above, Definition \ref{def:SDEIntervention} takes an SDE as an
argument and yields another SDE, in contrast to, for example, taking the
distribution of an SDE, and yielding another distribution. This
corresponds to how an intervention in the framework of SEMs, see \cite{MR2548166},
takes a SEM and returns another SEM, instead of taking a distribution
and yielding another distribution. This is a key point, and allows us
in Section \ref{sec:InterventionInterpret} to use SEMs and DAGs to
interpret Definition \ref{def:SDEIntervention}, and view SDEs as a
natural extension of SEMs to continuous time models. 

Let $V$ be a finite set, and let $E$ be a subset of $V\times V$. A
directed graph $G$ on $V$ is a pair $(V,E)$. We refer to $V$ as the vertex set,
and refer to $E$ as the edge set. Note that by this definition, there can
be at most one edge between any pair of vertices. A path is an unbroken series of
vertices and edges such that no vertices are repeated except possibly
the initial and terminal vertices. A directed cycle is a path with the same
initial and terminal vertices and all arrows pointing in the same direction. We say that $G$ is an acyclic directed
graph (DAG) if $G$ contains no directed cycles. Note that this in particular
excludes that the graph contains an edge with the same initial and terminal
vertex. For any graph $G$ and $i\in V$, we write $\pa(i) = \{j\in V\mid (j,i)\in E\}$, and refer to $\pa(i)$
as the parents of the vertex $i$. If we wish to emphazise the graph
$G$, we also write $\pa_G(i)$. 

A structural equation model (SEM) consists of three components:
\begin{enumerate}
\item Two families $(X_i)_{i\in V}$ and $(U_i)_{i\in V}$ of random variables.
\item A directed acyclic graph $G$ on $V$.
\item A set of functional relationships $X_i=f_i(X_{\pa_G(i)},U_i)$.
\end{enumerate}

We refer to $(X_i)_{i\in V}$ as the primary variables and $(U_i)_{i\in
  V}$ as the noise variables. Note that we do not a priori assume that the
noise variables are independent. The idea behind a SEM is that the DAG
provides the sequence in which the functional relationships are
evaluated, thus yielding an algorithm for obtaining the values of
$(X_i)_{i\in V}$ from $(U_i)_{i\in V}$. A SEM does not only yield the
distribution of the variables $(X_i)_{i\in V}$, but also a description
of a data generating mechanism. This is made precise by the notion
of an intervention, see Definition 3.2.1 of \cite{MR2548166}. Chapter
3 of \cite{MR2548166} discusses interventions where a subset of variables
are set to a constant value. We will need to consider a more general
type of interventions where variables are set to values depending on other
variables. Therefore, our definition below extends Definition 3.2.1 of
\cite{MR2548166}. See Chapter 4 of \cite{MR2548166} for more on this
type of interventions.



\begin{definition}
\label{def:SEMIntervention}
Consider a SEM with primary variables $(X_i)_{i\in V}$, noise
variables $(U_i)_{i\in V}$, DAG $G$ and functional relationships
$X_i=f_i(X_{\pa_G(i)},U_i)$. Let $A$ be a subset of $V$, and for $i
\in A$ let $I(i) \subseteq V \backslash A$ and $\zeta_i(X_{I(i)})$ be a function of the primary 
variables with indices in $I(i)$. We form a new graph $G'$ by replacing
$\pa_G(i)$ with $I(i)$ for $i \in A$. We assume that $G'$ is
a DAG. The postintervention SEM obtained by doing $X_i := \zeta_i(X_{I(i)})$ for $i\in A$
is a SEM with primary variables $(X_i)_{i\in V}$, noise
variables $(U_i)_{i\in V}$, DAG $G'$ and functional
relationships obtained by substituting all
occurrences of $X_i$ by $\zeta_i(X_{I(i)})$ for
$i \in A$. 
\end{definition}

In short, Definition \ref{def:SEMIntervention} describes the effect of
intervening and setting $X_i$ for $i\in A$ to be a function of certain
variables in $V\setminus A$. In the case where the $\zeta_i$
are constant, this reduces to Definition 3.2.1 of \cite{MR2548166}.

\section{Interpretation of postintervention SDEs}

\label{sec:InterventionInterpret}

In this section, we show that under Lipschitz conditions on the
coefficients in (\ref{eq:MainSDE}) and the intervention mapping, the solution to the postintervention SDE described in
Definition \ref{def:SDEIntervention} essentially is the limit of a sequence of
postintervention SEMs as described in Definition \ref{def:SEMIntervention} based on the Euler scheme of
(\ref{eq:MainSDE}). We use this to clarify the role of the driving
semimartingales $Z^1,\ldots,Z^d$. Also, we will use this result to prove the main theorem on
identifiability in Section \ref{sec:Identifiability}.

\begin{definition}
\label{def:Signature}
The signature of the SDE (\ref{eq:MainSDE}) is the graph $S$
with vertex set $\{1,\ldots,n\}$ and an edge from $i$ to $j$ if it holds
that there is $k$ such that the mapping $a_{jk}$ is not independent of the $i$'th coordinate.
\end{definition}

Letting $a_{j\cdot}=(a_{j1},\ldots,a_{jd})$, another way of describing the signature $S$ in Definition \ref{def:Signature}
is that there is an edge from $i$ to $j$ if $x_i\mapsto a_{j\cdot}(x)$ is
not constant, or equivalently, there is no edge from $i$ to $j$ if it holds
for all $k$ that $a_{jk}$ does not depend on the $i$'th coordinate.

\begin{definition}
\label{def:LocallyUnaffected}
We say that $X^j$ is locally unaffected by $X^i$ in the SDE
(\ref{eq:MainSDE}) if there is no edge from $i$ to $j$ in the signature of (\ref{eq:MainSDE}).
\end{definition}

Being locally unaffected is a property of two coordinates of an SDE. If
there is no risk of ambiguity, we leave out the SDE and simply state that
$X^j$ is locally unaffected by $X^i$.

The signature is used in the following definition to define a SEM corresponding to the Euler scheme for
(\ref{eq:MainSDE}). With a slight abuse of notation, we choose
in Definition \ref{def:EulerScheme} for convenience to consider the initial variables
$X^1_0,\ldots,X^p_0$ as primary variables, even though these variables have no associated noise variables in the SEM.
This is not a problem as it is nonetheless clear how interventions
for the SEM given in Definition \ref{def:EulerScheme} should be understood.

\begin{definition}
\label{def:EulerScheme}
Fix $T>0$ and consider $\Delta>0$ such that $T/\Delta$ is a natural number. Let
$N=T/\Delta$ and $t_k = k\Delta$. The Euler SEM over $[0,T]$ with
step size $\Delta$ for (\ref{eq:MainSDE}) consists of the following:
\begin{enumerate}
\item The primary variables are the $p(N+1)$ variables in the set
  $(X^\Delta_{t_k})_{0\le k\le N}$, indexed by $\{0,\ldots,N\}\times\{1,\ldots,p\}$.
\item For $1\le k \le N$, the noise variable for the $i$'th coordinate of $X^\Delta_{t_k}$ is
  the $d$-dimensional variable $Z_{t_k}-Z_{t_{k-1}}$.
\item The DAG is the graph $G=(V,E)$ with vertex set $\{0,\ldots,N\}\times\{1,\ldots,p\}$ defined by
having $((i_1,j_1),(i_2,j_2))$ be an edge of $D$ if and only if $i_2 =
i_1 + 1$ and either $j_2=j_1$ or $(j_1,j_2)$ is an edge in the signature of
(\ref{eq:MainSDE}).
\item The functional relationships are given by:
\begin{align}
  (X^\Delta_{t_k})^i &= (X^\Delta_{t_{k-1}})^i + \sum_{j=1}^d a_{ij}(X^\Delta_{t_{k-1}})(Z^j_{t_k}-Z^j_{t_{k-1}}).
\end{align}
\end{enumerate}
\end{definition}

A visualization of the DAG for the SEM of Definition \ref{def:EulerScheme}
is shown in Figure \ref{figure:HMDAG}. The figure shows
how the signature $S$ determines the DAG describing the algorithm
for calculating the variables in the Euler SEMs. Making the
constant intervention $X^1_{t_k}:=\zeta$ for all $k$ corresponds to removing
the top row in Figure \ref{figure:HMDAG}.

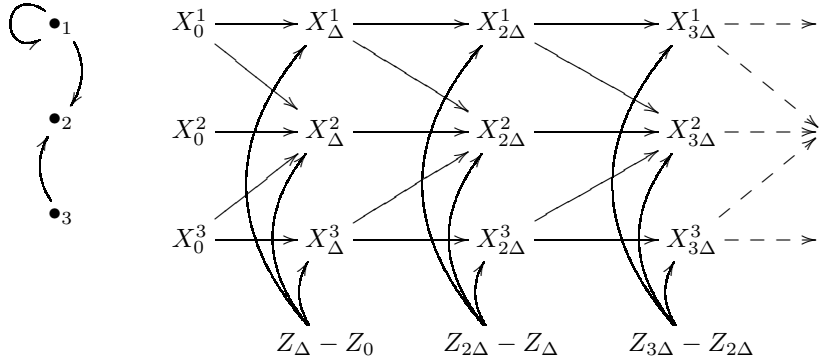
\begin{figure}[htb]
\begin{minipage}[b]{0.20\linewidth}
\centering
\begin{displaymath}
  \xymatrix@C=1cm{\bullet_1 \ar@(ul,dl)[] \ar@/^8pt/[d] \\ \bullet_2 \\ \bullet_3 \ar@/^8pt/[u] \\ }
\end{displaymath}
\end{minipage}
\begin{minipage}[b]{0.700\linewidth}
\centering
\begin{displaymath}
  \xymatrix@C=0.7cm{X^1_0\ar[r]\ar[dr]&X^1_{\Delta}\ar[r]\ar[dr]&X^1_{2\Delta}\ar[r]\ar[dr]&X^1_{3\Delta}\ar@{-->}[r]\ar@{-->}[dr]&\\
                    X^2_0\ar[r]       &X^2_{\Delta}\ar[r]       &X^2_{2\Delta}\ar[r]       &X^2_{3\Delta}\ar@{-->}[r]             &\\
                    X^3_0\ar[r]\ar[ur]&X^3_{\Delta}\ar[r]\ar[ur]&X^3_{2\Delta}\ar[r]\ar[ur]&X^3_{3\Delta}\ar@{-->}[r]\ar@{-->}[ur]&\\
                  & Z_{\Delta}-Z_0 \ar@/^10pt/[u] \ar@/^20pt/[uu] \ar@/^30pt/[uuu] & 
                  Z_{2\Delta}-Z_{\Delta} \ar@/^10pt/[u] \ar@/^20pt/[uu] \ar@/^30pt/[uuu] & 
                  Z_{3\Delta}-Z_{2\Delta} \ar@/^10pt/[u] \ar@/^20pt/[uu] \ar@/^30pt/[uuu] & 
                 }
\end{displaymath}
\end{minipage}
\caption{The signature for a three-dimensional SDE (left) and the DAG for the
  corresponding Euler SEM (right).}
\label{figure:HMDAG}
\end{figure}

Combining the following two lemmas yields the main result of this section.

\begin{lemma}
\label{lemma:EulerConvergence}
Assume that $a:\RR^p\to\MM(p,d)$ is Lipschitz. Fix $T>0$ and let $(\Delta_n)_{n\ge1}$ be a sequence of positive numbers converging to zero such
that $T/\Delta_n$ is natural for all $n\ge1$. For each $n$, there
exists a pathwisely unique solution to the equation
\begin{align}
  (X^n_t)^i &= X^i_0 + \sum_{j=1}^d \int_0^t a_{ij}(X^n_{\eta_n(s-)})\df{Z^j}_s,
  \qquad i\le p,\label{eq:EulerSDE}
\end{align}
where $\eta_n(t) = k\Delta_n$ for $k\Delta_n\le
t<(k+1)\Delta_n$, satisfying that $((X^n)_{t_k})_{0\le k\le T/\Delta_n}$ are
the primary variables in the Euler SEM
for (\ref{eq:MainSDE}), and $\sup_{0\le t\le T}|X_t-X^n_t|$
converges in probability to zero, where $X$ is the solution to (\ref{eq:MainSDE}).
\end{lemma}

\begin{proof}
By inspection, (\ref{eq:EulerSDE}) has a unique solution, and
$((X^n)_{t_k})_{k\le T/\Delta_n}$ is the primary variables in the
Euler SEM for (\ref{eq:MainSDE}). That $\sup_{0\le t\le T}|X_t-X^n_t|$
converges in probability to zero is the corollary to Theorem V.16 of
\cite{MR2273672}.
\end{proof}

\begin{lemma}
\label{lemma:EulerHMCIntervened}
Fix $T>0$ and consider $\Delta>0$ such that $T/\Delta$ is a natural
number. Fix $m\le p$ and $\zeta:\RR^{p-1}\to\RR$. The Euler SEM for the stochastic
differential equation (\ref{eq:IntervenedSDE}) is equal to the result
of removing
the $m$'th coordinate of the postintervention SEM obtained by the intervention
$(X^\Delta_{t_k})^m := \zeta((X^\Delta_{t_{k-1}})^{-m})$ for
$0\le k\le T/\Delta$ in the Euler SEM for (\ref{eq:MainSDE}).
\end{lemma}

\begin{proof}
The functional relationships in the Euler SEM for (\ref{eq:MainSDE}) are
\begin{align}
\label{eq:ObservationalEulerSEMFun}
  (X^\Delta_{t_k})^i &= (X^\Delta_{t_{k-1}})^i + \sum_{j=1}^d a_{ij}(X^\Delta_{t_{k-1}})(Z^j_{t_k}-Z^j_{t_{k-1}}),
\end{align}
while for (\ref{eq:IntervenedSDE}) and $i\neq m$, they are
\begin{align}
\label{eq:InterventionEulerSEMFun}
  (Y^\Delta_{t_k})^i &= (Y^\Delta_{t_{k-1}})^i + \sum_{j=1}^d b_{ij}((Y^\Delta_{t_{k-1}})^{-m})(Z^j_{t_k}-Z^j_{t_{k-1}})\notag\\
                     &= (Y^\Delta_{t_{k-1}})^i + \sum_{j=1}^d a_{ij}(Y^\Delta_{t_{k-1}})(Z^j_{t_k}-Z^j_{t_{k-1}}),
\end{align}
where $(Y^\Delta_{t_k})^m = \zeta((Y^\Delta_{t_{k-1}})^{-m})$. By inspection,
(\ref{eq:InterventionEulerSEMFun}) is the result of the stated intervention
in the Euler SEM according to Definition \ref{def:SEMIntervention}.
\end{proof}

Together, Lemma \ref{lemma:EulerConvergence} and Lemma
\ref{lemma:EulerHMCIntervened} states that the diagram
in Figure \ref{figure:DiscreteTimeIntervention} commutes: Defining
interventions directly in terms of changing the terms in the stochastic
differential equation has the same effect as intervening in the
Euler SEM and taking the limit.

\begin{figure}[htb]
\begin{center}
\begin{displaymath}
  \xymatrix@C=1cm{\fbox{\parbox{0.50\linewidth}{\center Euler SEM for observational SDE}} \ar[r] \ar[d]&
                  \fbox{\parbox{0.30\linewidth}{\center Observational SDE}} \ar[d] \\
                  \fbox{\parbox{0.50\linewidth}{\center Postintervention Euler SEM}} \ar[r] &
                  \fbox{\parbox{0.30\linewidth}{\center Postintervention SDE}} \\
                 }
\end{displaymath}
\end{center}
\caption{The interpretation of intervention in a stochastic
  differential equation understood as the limit of interventions in
  the Euler SEMs.}
\label{figure:DiscreteTimeIntervention}
\end{figure}
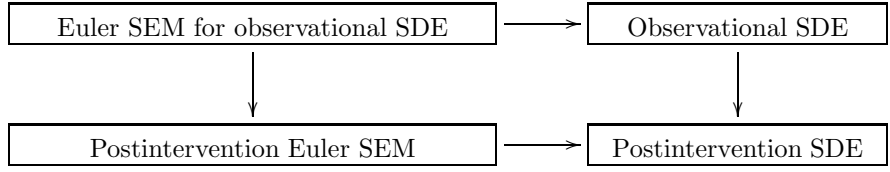

These results clarify what Definition \ref{def:SDEIntervention} means,
and in particular, when this generic definition of intervention is
applicable when background mechanisms are unknown, such as Example
\ref{example:Yeast}. The intuition behind Definition
\ref{def:SDEIntervention} is that interventions are assumed not
to influence the semimartingale $Z$ directly. This is made
concrete by assuming that the family $(Z_{t_k}-Z_{t_{k-1}})_{k\le N}$ are the
noise variables in the Euler SEM, such that there are no arrows in the DAG
for the SEM with terminal vertices in $(Z_{t_k}-Z_{t_{k-1}})_{k\le
  N}$. The lemmas show that when this condition holds true, the notion of intervention
given in Definition \ref{def:SDEIntervention} is consistent with the result of
intervention in the Euler SEM. Note that this does not constitute a proof of causality. Rather, it gives
guidelines as to when it is reasonable to expect that our notion of
intervention will reflect real-world interventions: namely, when none of
the coordinates $X^i$ have a direct effect on the driving
semimartingale $Z$. Whether this is the case or not is in general not a testable assumption.

Furthermore, note that the arrows across columns in the Euler SEM is
determined by the signature of the SDE. Therefore, if we accept the
hypothesis that the DAG of the Euler SEM describes the causal links between
the coordinates of the SDE, then the signature $S$ describes which coordinates of the SDE
(\ref{eq:MainSDE}) are causally dependent on each other in an
infinitesimal sense. Also note that as we are not using the Euler SEMs to draw any conclusions
about the distribution of the variables, we do not require independence of
the noise variables $(Z_{t_k}-Z_{t_{k-1}})_{k\le N}$. In particular,
the variables in the Euler SEM do not need to be Markov with respect to the DAG in the
sense of \cite{MR2548166}.

Concluding this section, we give an example to illustrate that the notion of
intervention given in Definition \ref{def:SDEIntervention}, and the
corresponding causal interpretation outlined above, may not always be
applicable.

\begin{example}
\label{example:DoubleVectorExample}
Let $X^1 = W$ be a one-dimensional
Wiener process, consider a twice continuously
differentiable function $f:\RR\to\RR$ and assume that for all $t\ge0$,
it holds that
\begin{align}
\label{eq:FunctionalOvertRelationship}
  X_t^2 &= f(X_t^1).
\end{align}
We now make the following assumption: Assume that
(\ref{eq:FunctionalOvertRelationship}) represents the actual causal
relationship between $X^1$ and $X^2$, in the sense that the result on
$X^2$ of the intervention  $X^1 := \zeta$ is the process
\begin{align}
\label{eq:FunctionalOvertRelationshipRes}
X^2_t &= f(\zeta).
\end{align}
Now, by It\^{o}'s lemma, it holds that
\begin{eqnarray}
\label{eq:itocausal}
X^2_t & = & f(X_0^1) + \frac{1}{2} \int_0^t
f''(X_s^1) \df{[X^1]_s} + \int_0^t f'(X_s^1) \df{X^1_s}\\
& = & f(0) + \frac{1}{2} \int_0^t
f''(X_s^1) \df{s} +  \int_0^t f'(X_s^1) \df{W_s},
\nonumber
\end{eqnarray}
such that $(X^1,X^2)$ satisfies
\begin{align}
  X^1_t &= \int_0^t \dv W_s \\
  X^2_t &= f(0) + \frac{1}{2} \int_0^t f''(X_s^1) \df{s} +  \int_0^t f'(X_s^1) \df{W_s},
\end{align}
which together yields a two-dimensional SDE of the form given in (\ref{eq:MainSDE}). Therefore, we may apply Definition
\ref{def:SDEIntervention} to this SDE. The resulting
postintervention SDE for $X^2$ under the intervention $X^1 := \zeta$ is
\begin{align}
  X^2_t &= f(0) + \frac{1}{2} \int_0^t f''(\zeta) \df{s} +  \int_0^t f'(\zeta) \df{W_s},
\end{align}
which yields the result $X^2_t = f(0) + \frac{1}{2}f''(\zeta) t + f'(\zeta) W_t$,
in contradiction with our assumed result in
(\ref{eq:FunctionalOvertRelationshipRes}), $X^2_t = f(\zeta)$. This
shows that we may conceptualize ideas about the effects of
interventions which are rather natural, but which are not captured by
Definition \ref{def:SDEIntervention}. This illustrates the importance of the
conclusions made above: We can only argue under certain circumstances that Definition
\ref{def:SDEIntervention} is a reasonable description of the effects
of intervention.
\hfill$\circ$
\end{example}

We note that in Example \ref{example:DoubleVectorExample}, it is not
the use of It\^{o}'s lemma which yields the discrepancy between the
results of Definition \ref{def:SDEIntervention} and the assumed
result, (\ref{eq:FunctionalOvertRelationshipRes}). Rather, it is the
subsequent substitution of $X^1$ by $W$. In fact, if we intervene
directly in (\ref{eq:itocausal}) by replacing $X^1$ by the constant
$\zeta$, the result would be that $X^2_t = f(\zeta)$, in accordance
with (\ref{eq:FunctionalOvertRelationshipRes}). However, Definition
\ref{def:SDEIntervention} does not allow for such interventions on
the integrators. To do so generally would complicate matters, and we
will not pursue this any further.

\section{Identifiability of postintervention distributions}


\label{sec:Identifiability}

In this section we formulate a result, Theorem \ref{theorem:LevyDiffusionIdentifiability}, giving conditions for the
postintervention distributions to be determined by
uniquely identifiable aspects of the SDE. We show that 
if the SDE is driven by a L{\'e}vy process, the postintervention
distribution is determined by the generator.

To introduce the generator associated with the SDE (\ref{eq:MainSDE}),
when it is driven by a  L{\'e}vy process, 
we need to introduce L{\'e}vy triplets. A L{\'e}vy measure on $\RR^d$
is a measure $\nu$ assigning zero measure to $\{0\}$ such that $x\mapsto\min\{1,\|x\|^2\}$ is integrable with respect to $\nu$. A
$d$-dimensional L{\'e}vy triplet is a triplet $(\alpha,C,\nu)$, where
$\alpha$ is an element of $\RR^d$, $C$ is a positive semidefinite $d\times
d$ matrix and $\nu$ is a L{\'e}vy measure on $\RR^d$. Recall that 
for any bounded neighborhood $D$ of zero in $\RR^d$ and any
$d$-dimensional L{\'e}vy process $X$, there is a L{\'e}vy triplet
$(\alpha,C,\nu)$ such that
\begin{align}
  \label{eq:charfun}
  Ee^{iu^tX_1}
  &=\exp\left(i u^t\alpha
            -\frac{1}{2}u^tCu-\int_{\RR^d}e^{i u^tx}-1-iu^tx 1_D(x)\dv \nu(x)\right),
\end{align}
and this triplet uniquely determines the distribution of $X$, see
Theorem 1.2.14 of \cite{MR2512800}. We refer to
$(\alpha,C,\nu)$ as the characteristics of $X$ with respect to $D$, or as the $D$-characteristic
triplet of $X$. Conversely, for any bounded neighborhood $D$ of zero in $\RR^d$ and
any L{\'e}vy triplet $(\alpha,C,\nu)$, there exists a L{\'e}vy process
having $(\alpha,C,\nu)$ as its $D$-characteristic triplet.

The generator of (\ref{eq:MainSDE}) is defined as a linear operator on the set
$C_0^2(\RR^p)$ of twice continuously differentiable functions such
that the function itself together with all its first and second partial
derivatives vanish at infinity. 

\begin{definition}
\label{def:LevyDiffusionGenerator}
Let $D$ be a bounded neighborhood of zero in $\RR^d$. Consider the SDE
(\ref{eq:MainSDE}), where $Z$ is a $d$-dimensional L{\'e}vy process with $D$-characteristic
triplet $(\alpha,C,\nu)$ and $a:\RR^p\to\MM(p,d)$. We define the
generator $A$ of (\ref{eq:MainSDE}) on $C_0^2(\RR^p)$ by 
\begin{align}
\label{eq:DBasedLevySDEGenerator}
  Af(x)
  &= \sum_{i=1}^p\sum_{j=1}^d a_{ij}(x)\alpha_j\frac{\partial f}{\partial x_i}(x)
   +\frac{1}{2}\sum_{i=1}^p\sum_{j=1}^p
                 (a(x)Ca(x)^t)_{ij}\frac{\partial^2
                   f}{\partial x_i\partial x_j}(x)\notag\\
  &+\int_{\RR^p} f(x+a(x)y)-f(x)-1_D(y)\sum_{i=1}^p \frac{\partial
    f}{\partial x_i}(x)\sum_{j=1}^d a_{ij}(x)y_j\dv \nu(y)
\end{align}
for  $f \in C_0^2(\RR^p)$ and $x \in \RR^p$. 
\end{definition}

It holds that for any choice of $a$ that the generator $A$ is well defined on
$C_0^2(\RR^p)$ with values in the set of functions on $\RR^p$. If
we are willing to put restrictions on $a$, the range of the generator 
can be restricted as well. 

The interest in the generator stems from the fact that when $Z$ is
a L{\'e}vy process, the generator of (\ref{eq:MainSDE}) can usually be determined by the
semigroup of transition probabilities for the Markov
process that solves (\ref{eq:MainSDE}). We state one such result 
here. Lemma \ref{lemma:SDESemigroupExistenceAndUniqueness} is a
folklore result, and follows from the results in Chapter 6 of \cite{MR2512800}.

\begin{lemma}
\label{lemma:SDESemigroupExistenceAndUniqueness} If $Z$ is a L{\'e}vy
process and  $a:\RR^p\to\MM(p,d)$ is Lipschitz and
bounded then there exists a unique Feller semigroup $(P_t)$ with the property that all
solutions of (\ref{eq:MainSDE}) are Feller processes with semigroup
$(P_t)$. Moreover, the generator $A$ of (\ref{eq:MainSDE}) satisfies that
\begin{align}
Af &= \lim_{t\to0}t^{-1}(P_tf-P_0f)
\end{align}
for $f \in C_0^2(\RR^p)$, where convergence is in the uniform norm on $C_0^2(\RR^p)$.
\end{lemma}

For a treatment of the theory of Markov processes and L{\'e}vy
processes, and in particular for notions such as
Feller processes, Feller semigroups, generators,
L{\'e}vy processes and so forth, see 
\cite{MR838085,MR2512800,MR1739520}. We are now ready to state our main result on identifiability.

\begin{theorem}
\label{theorem:LevyDiffusionIdentifiability}
Consider the SDEs
\begin{align}
  X^i_t &= X^i_0 + \sum_{j=1}^d \int_0^t a_{ij}(X_s)\df{Z^j}_s,
  \qquad i\le p,\label{eq:IdentifiabilitySDE1}
\end{align}
and
\begin{align}
  \tilde{X}^i_t &= \tilde{X}^i_0 + \sum_{j=1}^d \int_0^t \tilde{a}_{ij}(\tilde{X}_s)\df{\tilde{Z}^j}_s,
  \qquad i\le p,\label{eq:IdentifiabilitySDE2}
\end{align}
where $Z$ is a $d$-dimensional L{\'e}vy process and $\tilde{Z}$ is a
$\tilde{d}$-dimensional L{\'e}vy process. Assume that (\ref{eq:IdentifiabilitySDE1}) and
(\ref{eq:IdentifiabilitySDE2}) have the same
generator, that $a : \RR^p \to \MM(p, d)$ and $\zeta : \RR^{p-1} \to
\RR$ are Lipschitz and that the initial values have the same distribution. Then the postintervention distributions of doing $X^m:=\zeta(X^{-m})$ in
(\ref{eq:IdentifiabilitySDE1}) and doing $\tilde{X}^m:=\zeta(\tilde{X}^{-m})$ in
(\ref{eq:IdentifiabilitySDE2}) are equal for any choice of $\zeta$ and
$m$. 
\end{theorem}

Theorem \ref{theorem:LevyDiffusionIdentifiability} is proven in Appendix
\ref{sec:IdentifiabilityProof}. Theorem
\ref{theorem:LevyDiffusionIdentifiability} states that for SDEs with a
L{\'e}vy process as the driving semimartingale, postintervention
distributions are identifiable from the generator. In the remainder of this
section, we discuss the content of Theorem \ref{theorem:LevyDiffusionIdentifiability}.

First, recall that a main theme of the DAG-based framework for causal
inference as in \cite{MR1815675,MR2548166} is to identify conditions for when
postintervention distributions are identifiable from the observational
distribution. Theorem \ref{theorem:LevyDiffusionIdentifiability} gives a
criterion for when postintervention distributions are identifiable from the
generator of the SDE, which is not exactly the same. Nonetheless, in a
large family of naturally occurring cases, the semigroup is identifiable
from the observational distribution. This is for example the case if the solutions to
(\ref{eq:IdentifiabilitySDE1}) and (\ref{eq:IdentifiabilitySDE2}) are
irreducible, as the family of transition probabilities
in this case will be identifiable from the observational distribution,
allowing us to obtain the generator through Lemma \ref{lemma:SDESemigroupExistenceAndUniqueness}.

Next, we comment on the relationship between the result of Theorem
\ref{theorem:LevyDiffusionIdentifiability} and identifiability results of
DAG-based causal inference. Consider the Euler SEM of Definition
\ref{def:EulerScheme}, illustrated in Figure \ref{figure:HMDAG}. In the DAG
of this SEM, the orientation of all arrows is assumed known: All
orientations for arrows from primary variables point forward in time. If the
error variables for each primary variable were independent, it would hold
that the distribution of the variables would be Markov with respect to the
DAG in the sense of \cite{MR2548166}. In this case, by the results of \cite{VP}, we would be able
to identify the skeleton of the graph (that is, its undirected edges) from
the observational distribution. As all orientations are given,
this leads to identifiability of the entire graph. Using the truncated
factorization (3.10) of \cite{MR2548166}, this leads to identifiability of intervention distributions
from the observational distribution. Thus, in this case, identifiability
would not be a surprising result.


However, when the driving semimartingale $Z$ is a
L{\'e}vy process, the error variables are independent across time, but are
not independent across coordinates: For each $k$,
the variables $X^1_{\Delta k},\ldots,X^p_{\Delta k}$ have the same
$d$-dimensional error variable, namely $Z_{\Delta
  k}-Z_{\Delta(k-1)}$, and so the Euler SEM illustrated in Figure
\ref{figure:HMDAG} is not Markov with respect to its DAG. Therefore, our
scenario differs from the conventional causal modeling scenario of \cite{MR2548166}
in two ways: Both by considering a continuous-time model with uncountably
many variables and by considering a particular type of dependent errors. 

We end the section with three examples. Example
\ref{ex:OUIdentifiabilityExplicit} considers a particularly simple
scenario where identifiability of postintervention distributions can
be seen explicitly from the transition probabilities. In Example
\ref{ex:TwoSignatures}, we show that it is possible for two SDEs with
the same distribution to have different signatures. Remarkably, this shows that
while postintervention distributions are identifiable by Theorem
\ref{theorem:LevyDiffusionIdentifiability}, the signature of the true
SDE is not generally identifiable. However, we expect that the
behaviour observed in Example \ref{ex:TwoSignatures} is atypical,
similarly to the absence of faithfulness in Gaussian SEMs, see Theorem 3.2 of
\cite{MR1815675}. Finally, in Example
\ref{example:YeastConclusion}, we show how Theorem
\ref{theorem:LevyDiffusionIdentifiability} allows us to infer
intervention effects of knocking out genes in our previous example on
\textit{S. Cerevisiae}, Example \ref{example:Yeast}.

\begin{example}
\label{ex:OUIdentifiabilityExplicit}
Let $W$ and $\tilde{W}$ be $d$-dimensional and $\tilde{d}$-dimensional
Brownian motions, let $B$ and $\tilde{B}$ be $p\times p$ matrices, and let $\sigma$ and
$\tilde{\sigma}$ be $p\times d$ and $p\times\tilde{d}$ matrices. Consider two processes $X$ and $Y$ being the unique solutions to the
Ornstein-Uhlenbeck SDEs
\begin{align}
  X_t &= X_0 + \int_0^t BX_t\dv t + \sigma W_t
\end{align}
and
\begin{align}
  Y_t &= Y_0 + \int_0^t \tilde{B}X_t\dv t + \tilde{\sigma} W_t.
\end{align}
We will show by a direct analysis that if the generators of the SDEs
are equal and the initial distributions are the same, then the postintervention distributions are equal as well. For
notational simplicity, we consider intervening on the first coordinate,
making the interventions $X^1:=\zeta$ and $Y^1:=\zeta$. It will suffice to show
equality of distributions for the non-intervened coordinates in the
postintervention distributions. Consider block decompositions of the form
\begin{align}
  B = \left(\begin{array}{cc}
            B_{11} & B_{12} \\
            B_{21} & B_{22}
            \end{array}
      \right)\qquad \textrm{and} \qquad
  \sigma = \left(\begin{array}{c}
            \sigma_1 \\
            \sigma_2
            \end{array}
      \right),
\end{align}
where $B_{11}$ is a $1\times 1$ matrix and $B_{22}$ is a $(p-1)\times(p-1)$
matrix and $\sigma_1$ is a $1\times d$ matrix and $\sigma_2$ is a $(p-1)\times d$
matrix. Also consider corresponding decompositions of $\tilde{B}$ and $\tilde{\sigma}$. 

Assume that the generators of the SDEs are equal, and assume that $X_0$ and
$Y_0$ have the same distribution. The transition
probabilities for $X$ and $Y$ are then the same. With $P_t(x,\cdot)$
denoting the transition probability of moving from state $x$ in time $t$
for $X$, the results of \cite{MJ} show that
\begin{align}
  P_t(x,\cdot) =
  \NNN\left(\exp(tB)x,\int_0^t\exp(sB)\sigma\sigma^t\exp(sB^t)\dv s\right),
\end{align}
where the right-hand side denotes a Gaussian distribution, and similarly for the transition probabilities of
$Y$. As these are equal for all $x\in\RR^p$ and $t\ge0$, we obtain $\exp(tB)=\exp(t\tilde{B})$ for all $t\ge0$, so by differentiating,
$B=\tilde{B}$ as well. Likewise, as $\int_0^t\exp(sB)\sigma\sigma^t\exp(sB^t)\dv
s=\int_0^t\exp(s\tilde{B})\tilde{\sigma}\tilde{\sigma}^t\exp(s\tilde{B}^t)\dv
s$ for all $t\ge0$, we obtain
$\sigma\sigma^t=\tilde{\sigma}\tilde{\sigma}^t$. Note that
\begin{align}
  \sigma\sigma^t &= \left(\begin{array}{cc}
                    \sigma_1\sigma_1^t & \sigma_1\sigma_2^t \\
                    \sigma_2\sigma_1^t & \sigma_2\sigma_2^t
                    \end{array}\right),
\end{align}
and similarly for $\tilde{\sigma}\tilde{\sigma}^t$. Therefore, we obtain in
particular that $\sigma_2\sigma_2^t = \tilde{\sigma}_2\tilde{\sigma}_2^t$.

Now, applying Definition
\ref{def:SDEIntervention} and recalling Example
\ref{example:OUInterExample}, the intervened processes minus the first
coordinate, $\tilde{X}^{-1}$ and $\tilde{Y}^{-1}$ (note that the
superscripts do not denote reciprocals), are
Ornstein-Uhlenbeck processes with initial values $X_0^{-1}$ and
$Y_0^{-1}$, mean reversion speeds $B_{22}$ and $\tilde{B}_{22}$, mean
reversion levels $-B_{22}^{-1}B_{21}\zeta$ and $-\tilde{B}_{22}^{-1}\tilde{B}_{21}\zeta$ and diffusion matrices
$\sigma_2$ and $\tilde{\sigma}_2$. As we above concluded that $X_0$ and
$Y_0$ have the same distribution, $B=\tilde{B}$
and $\sigma_2\sigma_2^t = \tilde{\sigma}_2\tilde{\sigma}_2^t$, we obtain
that the distributions of $\tilde{X}^{-1}$ and $\tilde{Y}^{-1}$ must be
equal. Thus, by direct calculation of transition probabilities,
we see that for the Ornstein-Uhlenbeck with zero mean reversion level,
intervention distributions are identifiable from the observational
distribution.
\hfill$\circ$
\end{example}

\begin{example}
\label{ex:TwoSignatures}
Consider the mapping $a:\RR^2\to\MM(2,2)$ defined by 
\begin{align}
  a(x) &= \left[\begin{array}{cc}
                x_1 & 0 \\ x_2^2/\sqrt{x_1^2+x_2^2} & -x_1x_2/\sqrt{x_1^2+x_2^2}
                \end{array}\right]
\end{align}
whenever $x$ is not zero, and $a(0)=0$. This mapping satisfies
\begin{align}
\label{ex:TwoSigTransProd}
  a(x)a(x)^t
  &=  \left[\begin{array}{cc}
                x_1 & 0 \\ x_2^2/\sqrt{x_1^2+x_2^2} & -x_1x_2/\sqrt{x_1^2+x_2^2}
                \end{array}\right]
      \left[\begin{array}{cc}
                x_1 & x_2^2/\sqrt{x_1^2+x_2^2} \\ 0 & -x_1x_2/\sqrt{x_1^2+x_2^2}
                \end{array}\right]\notag\\
  &=\left[\begin{array}{cc}
                x_1^2 & x_1x_2^2/\sqrt{x_1^2+x_2^2} \\
                x_1x_2^2/\sqrt{x_1^2+x_2^2} & x_2^2
                \end{array}\right]
\end{align}
whenever $x\neq0$. We will construct another mapping $\tilde{a}$ which has a different
signature from $a$, but which has the same cross product as $a$,
in the sense of having $\tilde{a}(x)\tilde{a}(x)^t=a(x)a(x)^t$. To do so, define $p:\RR^2\to\MM(2,2)$ by
\begin{align}
  p(x) &= \frac{1}{\sqrt{x_1^2+x_2^2}}\left[\begin{array}{cc}
                x_1 & x_2  \\
                x_2 & -x_1
                \end{array}\right],
\end{align}
for $x\neq0$ and let $p(0)$ be the identity matrix. Put $\tilde{a}(x)=a(x)p(x)$. We then
obtain $\tilde{a}(0)=a(0)=0$ and
\begin{align}
  \tilde{a}(x)
  &= \frac{1}{\sqrt{x_1^2+x_2^2}}
     \left[\begin{array}{cc}
                x_1 & 0 \\ x_2^2/\sqrt{x_1^2+x_2^2} & -x_1x_2/\sqrt{x_1^2+x_2^2}
                \end{array}\right]
     \left[\begin{array}{cc}
           x_1 & x_2  \\
           x_2 & -x_1
     \end{array}\right]\notag\\
  &= \frac{1}{\sqrt{x_1^2+x_2^2}}\left[\begin{array}{cc}
           x_1^2 & x_1x_2  \\
           0 & (x_2^3+x_1^2x_2)/\sqrt{x_1^2+x_2^2}
     \end{array}\right]\notag\\
  &= \left[\begin{array}{cc}
           x_1^2/\sqrt{x_1^2+x_2^2} & x_1x_2/\sqrt{x_1^2+x_2^2}  \\
           0 & x_2
     \end{array}\right].
\end{align}
Note that the first row of $a$ depends only on the first coordinate,
while the second row depends on both coordinates. On the other hand,
the first row of $\tilde{a}$ depends on both coordinates, while the
second row of $\tilde{a}$ depends only on the second coordinate. This
translates into $a$ and $\tilde{a}$ corresponding to different
signatures, shown in Figure \ref{figure:TwoSignaturesEx}.

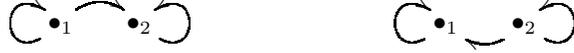
\begin{figure}[htb]
\begin{minipage}[b]{0.400\linewidth}
\vspace{0.5cm}
\centering
\begin{displaymath}
  \xymatrix@C=0.5cm{\bullet_1 \ar@(dl,ul)[] \ar@/^8pt/[r] & \bullet_2 \ar@(dr,ur)[] }
\end{displaymath}
\vspace{0.5cm}
\end{minipage}
\begin{minipage}[b]{0.400\linewidth}
\vspace{0.5cm}
\centering
\begin{displaymath}
  \xymatrix@C=0.5cm{\bullet_1 \ar@(dl,ul)[] & \bullet_2
    \ar@(dr,ur)[] \ar@/^8pt/[l] }
\end{displaymath}
\vspace{0.5cm}
\end{minipage}
\caption{Left: The signature corresponding to $a$. Right: The
  signature corresponding to $\tilde{a}$.}
\label{figure:TwoSignaturesEx}
\end{figure}

As $p(x)$ is orthonormal for all $x$, it holds that
$\tilde{a}(x)\tilde{a}(x)^t=a(x)a(x)^t$ and so the solutions to the two SDEs
\begin{align}
  \dv X_t &= a(X_t)\dv W_t \label{eq:TwoSigEx1}\\  
  \dv X_t &= \tilde{a}(X_t)\dv W_t \label{eq:TwoSigEx2}
\end{align}
have the same distribution. Thus, we have explicitly constructed two
SDEs with the same solution distributions but with different
signatures. Note now that the intervention $X^2:=\zeta$ in
(\ref{eq:TwoSigEx1}) yields an SDE where the first coordinate
satisfies
\begin{align}
\label{eq:IntervenedTwoSigEx1}
  \dv X^1_t &= X^1_t\dv W^1_t
\end{align}
while the intervention $X^2:=\zeta$ in
(\ref{eq:TwoSigEx2}) yields an SDE where the first coordinate
satisfies
\begin{align}
\label{eq:IntervenedTwoSigEx2}
  \dv X^1_t &= \frac{(X^1_t)^2}{\sqrt{(X^1_t)^2+\zeta^2}}\dv W^1_t
              +\frac{X^1_t\zeta}{\sqrt{(X^1_t)^2+\zeta^2}}\dv W^2_t 
\end{align}
The distribution of the solution to (\ref{eq:IntervenedTwoSigEx2}) is
a Markov process whose generator on $C^2_0(\RR)$ is given by
\begin{align}
  Af(x) &= \frac{x^4+(x\zeta)^2}{x^2+\zeta^2}\frac{\dv^2f}{\dv x^2}(x)
        = x^2\frac{\dv^2f}{\dv x^2}(x),
\end{align}
which is the generator of a geometric Brownian motion with zero drift.
This is the same as the generator of the solution to
(\ref{eq:IntervenedTwoSigEx1}). Thus, as required in Theorem
\ref{theorem:LevyDiffusionIdentifiability}, the postintervention
distributions are the same, even in this case where the signatures are different.
\hfill$\circ$
\end{example}



Example \ref{ex:TwoSignatures} illustrates a rather curious fact:
For SDE models, the postintervention distributions are identifiable
from the observational distribution, even when the signature and thus
the resulting DAGs of the Euler SEMs are not identifiable from the observational
distribution. One interpretation of this is that for SDEs, the
postintervention distributions will be the same for all signatures and
thus all resulting DAGs which are compatible with the observational
distribution. From this perspective, and in concordance with Theorem
\ref{theorem:LevyDiffusionIdentifiability}, the agreement of the two
postintervention distributions in Example \ref{ex:TwoSignatures} is
not so much related to the dependence structure of $a(x)$, but rather
on the dependence structure of $a(x)a(x)^t$, or equivalently,
$\tilde{a}(x)\tilde{a}(x)^t$.

This also indicates that in order to obtain a successful theory of
causality for SDEs, the relevant concept to consider is
postintervention distributions, and not the signatures, since the
latter is identifiable from the observational distribution while the
former is not. This contrasts with the classical DAG-based case, where
a natural methodology consists of first identifying the DAGs
compatible with the observational distribution and then, in order to
partially infer intervention effects, consider the intervention effect
for each possible DAG, as in \cite{MR2549555}.

\begin{example}
\label{example:YeastConclusion}
Consider again the yeast microorganism \textit{S.
  Cerevisiae}. In Example \ref{example:Yeast}, we assumed
that we were given observations $X_{t_0},\ldots,X_{t_p}$ over time of all $p =
5361$ genes of a non-mutant specimen of \textit{S. Cerevisiae}, and
we modeled these observations using an Ornstein-Uhlenbeck process, given by
\begin{align}
\label{eq:exampleOUForGenesConc}
  \dv X_t &= B(X_t - A)\dv t + \sigma \dv W_t.
\end{align}
In Example \ref{example:OUInterExample}, we used Definition
\ref{def:SDEIntervention} to calculate postintervention distributions
from Ornstein-Uhlenbeck processes. We concluded that if Definition
\ref{def:SDEIntervention} is applicable (a non-testable hypothesis,
according to the discussion of Section \ref{sec:InterventionInterpret}) and we could
identify the parameters in (\ref{eq:exampleOUForGenesConc}), then it
would be possible to draw inferences about the SDEs resulting from
interventions such as knocking out a single gene by setting $X^m:=0$. We also concluded that the parameter $\sigma$ is not identifiable in
our model. Thus, we cannot identify the true SDE, and thus cannot
identify the true postintervention SDE.

Assume now, however, that we are satisfied by only knowing the
distributional effects of interventions, corresponding to the
postintervention distribution. In this case, Theorem
\ref{theorem:LevyDiffusionIdentifiability} states that we are in fact
capable of inferring the postintervention distribution from the
observational distribution. One way of understanding this is that for
all SDEs with the same distribution as the observational distribution,
the postintervention distribution will be the same. This can be seen
explicitly in Example \ref{ex:OUIdentifiabilityExplicit}.

This conclusion allows us, for example, to infer the
effects of knocking out genes only from observational distributions.
\hfill$\circ$
\end{example}

\section{Discussion}

\label{sec:Discussion}

In this section, we will reflect on the results of the preceeding
sections and discuss opportunities for further work.

The definition of the postintervention SDE, Definition
\ref{def:SDEIntervention}, is a natural way to define how
interventions should affect stochastic dynamic systems. It constitutes
a generic notion of intervention effects applicable in cases such as
Example \ref{example:Yeast} where the background mechanisms of the
system are not known and the statistical model is primarily based on
observational data. However, the definition reflects assumptions about
the underlying causal nature of the system being modeled, and it is important to make precise
when the definition can be assumed to reflect an actual real-world
intervention and when the definition is simply a mathematical
construct. This is clarified in Section \ref{sec:InterventionInterpret},
where we used the DAG-based intervention calculus to show that the postintervention SDE of Definition
\ref{def:SDEIntervention} can be assumed to reflect real-world
interventions when the following hold:
\begin{enumerate}
\item The SDE reflects a data-generating mechanism in which
  the variables at a given timepoint are obtained as a function of the
  previous timepoints and the driving semimartingales.
\item The driving semimartingales are not directly affected by
  interventions, in the sense that they can be taken to be noise variables in the Euler SEMs.
\end{enumerate}

In full generality, causal mechanisms of a model are not
identifiable from the observational distribution, see
\cite{VP}. However, when considering only restricted classes of
structural equation models, the underlying causal mechanisms may often
be identifiable, see for example \cite{ZH,HJM,PB}. In such cases, linearity of the functional relationships or
Gaussianity of the noise variables often determine identifiability. In
our case, as shown by Theorem
\ref{theorem:LevyDiffusionIdentifiability}, identifiability holds
whenever the driving semimartingale is a L{\'e}vy process. This is a
key result for the applicability of our notion of intervention, and
yields the line of inference depicted in
Figure \ref{figure:LineOfCausality}: Under sufficient regularity
conditions such as appropriate notions of irreducibility, the generator of
a Markov process is identifiable from the observational distribution,
and Theorem \ref{theorem:LevyDiffusionIdentifiability} allows for
deducing postintervention distributions from the generator.

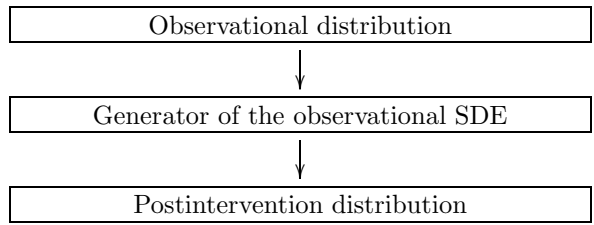
\begin{figure}[htb]
\begin{center}
\begin{displaymath}
  \xymatrix@R=0.5cm@C=1cm{\fbox{\parbox{0.60\linewidth}{\centering Observational distribution}} \ar[d]\\
                  \fbox{\parbox{0.60\linewidth}{\centering Generator of the observational SDE}} \ar[d] \\
                  \fbox{\parbox{0.60\linewidth}{\centering Postintervention distribution}} \\
                 }
\end{displaymath}
\end{center}
\caption{Line of inference for causality in SDEs.}
\label{figure:LineOfCausality}
\end{figure}

As argued in the series of examples comprised by Example
\ref{example:Yeast}, Example \ref{example:OUInterExample} and Example
\ref{example:YeastConclusion}, in the case of for example
time-dependent observations of gene expression data, this allows for
identification of knockout effects of genes using only
observational data.

The proofs given in Section \ref{sec:Identifiability} use
the Markov structure of the solution to the SDE. In the case where the
driving semimartingale has independent, but not stationary,
increments, the solution to the SDE will be an inhomogeneous Markov process, thus
also amenable to operator methods, though requiring more powerful
technical results. We expect that Theorem
\ref{theorem:LevyDiffusionIdentifiability} extends to this
case. It should also be noted that identifiability holds independently of the
dimension of the driving L{\'e}vy process. This is useful, for instance, in
relation to Example \ref{example:MotivatingSDE}. We do not need to 
use the specific SDE driven by a four-dimensional Wiener process. We
can replace the diffusion term in the SDE by a term involving the positive
definite square root of the diffusion matrix and a two-dimensional
Wiener process without affecting the postintervention distribution.

It is, however, important to be careful about the interpretation of the
identifiability result. The result states that when using Definition
\ref{def:SDEIntervention} to model interventions, the postintervention
distributions are identifiable. As discussed above, Definition
\ref{def:SDEIntervention} is not always useful as a notion of
intervention: This requires that we are willing to interpret the SDE
in a particular way. As Example \ref{example:DoubleVectorExample}
shows, not all SDEs are amenable to such an interpretation. This requires 
separate arguments, such as in Example \ref{example:MotivatingSDE}.

We also remark on the connection between our notion of intervention
and the framework of weak conditional local independence (WCLI) discussed in
\cite{MR2749916,MR2575938}. Definition 2 of \cite{MR2575938} defines
WCLI for semimartingales in the class $\DDD'$. In Remark 1 of
\cite{MR2575938}, it is explained how WCLI can lose its interpretation
if extended to larger classes of semimartingales. However,
the definition does make sense for all special semimartingales. Extending it to this class, let $X$ be the solution to an SDE of the
type (\ref{eq:MainSDE}), driven by a L{\'e}vy process and assume that
$X$ is a special semimartingale. One relationship between our notion
of intervention and WCLI is then this: It holds that if $X^i$ is
locally unaffected by $X^m$ in (\ref{eq:MainSDE}), then $X^i$ is WCLI
of $X^m$. This follows by considering the semimartingale
characteristics of solutions to SDEs, see for example Proposition
IX.5.3 of \cite{MR1943877}.

Our results offer opportunities for further research. One main
opportunity concerns latent variables: In the DAG-based framework of
\cite{MR2548166}, the back-door and front-door criteria shows how to
calculate intervention effects from the observational distribution in
the presence of latent variables. For an SDE, the causal structure is
summarized in the signature, see Definition \ref{def:Signature}, which
does not need to be acyclic, reflecting the possibility of
feedback loops. It is an open question how to obtain similar results
in terms of the signature in the case of, for example, a diffusion model
with some coordinates being unobserved. Another question concerns
criteria for when the signature contains useful high-level
information about the effects of interventions. As Example
\ref{ex:TwoSignatures} shows, this is not always the
case. We expect that the behaviour seen in Example
\ref{ex:TwoSignatures} is atypical, but have not shown any precise
results about this.

\appendix
\section{Proof of Theorem \ref{theorem:LevyDiffusionIdentifiability}}
\label{sec:IdentifiabilityProof}

In this appendix, we prove Theorem
\ref{theorem:LevyDiffusionIdentifiability}. We first state 
some well known and some simple results. The first result, Lemma
\ref{lem:markovintervention}, is an elementary yet crucial result about
interventions in discrete time Markov chains, allowing us to use the Euler scheme to prove
 Theorem \ref{theorem:LevyDiffusionIdentifiability}. Two additional
 lemmas are simple facts about generators. We do not give full proofs, but we
do briefly state how to use results from the literature to obtain
full proofs.

For any $G : \RR^p \times \RR^d \to \RR^p$ and $\zeta:\RR^{p-1}\to\RR$ we introduce 
$H_G : \RR^{p-1} \times \RR^d \to \RR^{p-1}$ by 
\begin{align}
  H_G(y,u)^i &= G((y_1,\ldots,\zeta(y),\ldots,y_p),u)^i
\end{align}
for $i \neq m$ with $\zeta(y)$ at the $m$'th position. Now if
$U$ and $V$ are random variables with values in $\RR^d$ and
$\RR^{d'}$, respectively, and if  $G : \RR^p \times \RR^d \to \RR^p$
and $\tilde{G} : \RR^p \times \RR^{d'} \to \RR^p$ fulfill that for all
$x \in \RR^p$ 
\begin{align}
  G(x, U)&\overset{\mathcal{D}}{=}\tilde{G}(x, V),
\end{align}
meaning that the variables are equal in distribution, then obviously
\begin{align}
H_G(y, U)&\overset{\mathcal{D}}{=} H_{\tilde{G}}(y, V).
\end{align}
for all $y \in \mathbb{R}^{p-1}$. The important consequence that we
can derive from this observation is that 
postintervention distributions in discrete-time Markov processes can
be identified from their transition distributions. Specifically,
consider the Markov processes
\begin{align}
  X_n &= G(X_{n-1}, U_n),\\
  \tilde{X}_n &= \tilde{G}(\tilde{X}_{n-1}, V_n),
\end{align}
defined recursively in terms of update functions $G$ and $\tilde{G}$ 
and sequences $(U_n)$ and $(V_n)$ of independent random variables with values in $\RR^d$ and
$\RR^{d'}$, respectively. We also introduce the corresponding intervened 
processes 
\begin{align}
Y_n &= H_G(Y_{n-1}, U_n),\\
\tilde{Y}_n &= H_{\tilde{G}}(\tilde{Y}_{n-1}, V_n).
\end{align}
The following lemma is a simple consequence of the considerations
above.

\begin{lemma}
\label{lem:markovintervention}
If $X_0 \overset{\mathcal{D}}{=} X_0$ and if 
\begin{align}
\label{eq:distid}
G(x, U_n) &\overset{\mathcal{D}}{=} \tilde{G}(x, V_n)
\end{align}
for all $n \geq 1$ and $x
\in \RR^p$, then $(Y_n)$ and $(\tilde{Y}_n)$ have the same
distribution.
\end{lemma}

Proving Theorem \ref{theorem:LevyDiffusionIdentifiability} via the
Euler scheme will be done by showing that the update formulas for the
Euler schemes for two processes with the same generator satisfy
(\ref{eq:distid}). The following lemma shows how to write
(\ref{eq:DBasedLevySDEGenerator}) of Definition
\ref{def:LevyDiffusionGenerator} in a form which is more suitable
for the subsequent proof.

\begin{lemma}
\label{lemma:LevyDiffusionGeneratorRewrite}
Let $E$ be a neighborhood of zero in $\RR^p$. On $C_0^2(\RR^p)$, the
generator (\ref{eq:DBasedLevySDEGenerator}) of the SDE (\ref{eq:MainSDE}) may be rewritten as
\begin{align}
\label{eq:LevySDEGenerator}
  Af(x)
  &= \sum_{i=1}^p\beta_i(x)\frac{\partial f}{\partial x_i}(x)
  +\frac{1}{2}\sum_{i=1}^p\sum_{j=1}^p
                 (a(x)Ca(x)^t)_{ij}\frac{\partial^2
                   f}{\partial x_i\partial x_j}(x)\notag\\
  &+\int_{\RR^p} f(x+y)-f(x)-1_E(y)\sum_{i=1}^p \frac{\partial
    f}{\partial x_i}(x)y_i\dv T^{a(x)}(\nu)(y),
\end{align}
where $T^{a_0}:\RR^d\to\RR^p$ is defined by $T^{a_0}(y)=a_0y$ for $a_0\in\MM(p,d)$, and
\begin{align}
\label{eq:LevySDEGeneratorRewriteBeta}
  \beta_i(x) &= \sum_{j=1}^d a_{ij}(x) \alpha_j+\int_{\RR^d} (1_{E}(T^{a(x)}(y))-1_D(y))\sum_{j=1}^d a_{ij}(x) y_j\dv \nu(y),
\end{align}
whenever the integrals are well defined and finite. This finiteness condition is in
particular satisfied if $E$ is bounded.
\end{lemma}

The proof of Lemma \ref{lemma:LevyDiffusionGeneratorRewrite} is elementary, as
(\ref{eq:DBasedLevySDEGenerator}) and (\ref{eq:LevySDEGenerator}) are equal for all $E$ such that the integrals in
(\ref{eq:LevySDEGenerator}) and (\ref{eq:LevySDEGeneratorRewriteBeta})
are well defined and finite. In the case where $E$ is bounded, this is seen to
be the case by the integrability properties of L{\'e}vy measures.

As a final preparation, we state a lemma on identity of two functionals on
$C_0^2(\RR^p)$. The nontrivial implication of the lemma is proving
that all coefficients are equal if only the functionals are the
same. 
\begin{lemma}
\label{lemma:LevyTypeGeneratorUniqueness}
Fix $x\in\RR^p$ and let $D$ be a bounded neighborhood of zero in $\RR^p$. Let $a,\tilde{a}\in\RR^p$ and
$b,\tilde{b}\in\MM(p,p)$, and let $\nu$ and
$\tilde{\nu}$ be two measures on $\RR^p$ such that $x\mapsto\min\{1,\|x\|^2\}$
is integrable with respect to $\nu$ and $\tilde{\nu}$. Consider two linear functionals
$A$ and $\tilde{A}$ from $C^2_0(\RR^p)$ to $\RR$, where $A$ is given by
\begin{align}
  Af &= \sum_{i=1}^p a_i \frac{\partial f}{\partial x_i}(x)
      +\frac{1}{2}\sum_{i=1}^p\sum_{j=1}^p b_{ij} \frac{\partial^2
        f}{\partial x_i\partial x_j}(x)\notag\\
      &+\int_{\RR^p} f(x+y)-f(x)-1_D(y) \sum_{i=1}^p\frac{\partial f}{\partial
        x_i}(x)y_i\dv \nu(y),
\end{align}
and $\tilde{A}$ is given by the same expression, with $\tilde{a}$,
$\tilde{b}$ and $\tilde{\nu}$ substituted for $a$, $b$ and $\nu$. It then holds that
$A=\tilde{A}$ if and only if $a=\tilde{a}$, $b=\tilde{b}$ and
$\nu=\tilde{\nu}$ on $\RR^p\setminus\{0\}$.
\end{lemma}


The proof of Lemma \ref{lemma:LevyTypeGeneratorUniqueness} can be obtained as follows. In the notation of the lemma,
assume that $A=\tilde{A}$. Using approximate units such as defined in
\cite{MR2453959}, prove that $\nu$ and $\tilde{\nu}$ agree on all sets
of the form $B^c$ where $B$ is a bounded neighborhood of zero. This implies $\nu=\tilde{\nu}$ on
$\RR^p\setminus\{0\}$. From this, $a=\tilde{a}$ and $b=\tilde{b}$ follows.

Using the  above, we may now make short work of the proof of Theorem
\ref{theorem:LevyDiffusionIdentifiability}.

\textit{Proof of Theorem \ref{theorem:LevyDiffusionIdentifiability}.}
Fix a bounded neighborhood $D$ of zero in $\RR^d$, a bounded neighborhood
$\tilde{D}$ of zero in $\RR^{\tilde{d}}$ and a bounded neighborhood $E$ of
zero in $\RR^p$. Assume that $Z$ has $D$-characteristics $(\alpha,C,\nu)$ and
that $\tilde{Z}$ has $\tilde{D}$-characteristics
$(\tilde{\alpha},\tilde{C},\tilde{\nu})$. For $a_0 \in \MM(p,d)$ define
$T^{a_0}:\RR^d\to\RR^p$ by $T^{a_0}(y)=a_0y$. Also define
\begin{align}
  \beta_i(x) &= \sum_{j=1}^d a_{ij}(x) \alpha_j+\int (1_{E}(a(x)y)-1_D(y))\sum_{j=1}^d a_{ij}(x) y_j\dv \nu(y)\\
  \tilde{\beta}_i(x) &= \sum_{j=1}^{\tilde{d}} \tilde{a}_{ij}(x) \tilde{\alpha}_j+\int (1_{E}(\tilde{a}(x)y)-1_{\tilde{D}}(y))\sum_{j=1}^{\tilde{d}} \tilde{a}_{ij}(x) y_j\dv \tilde{\nu}(y).
\end{align}
Let $A:C_0^2(\RR^p)\to C_0(\RR^p)$ be given by
(\ref{eq:LevySDEGenerator}), and let $\tilde{A}:C_0^2(\RR^p)\to
C_0(\RR^p)$ be given similarly, except with $\beta$, $a$, $C$, $\nu$,
$D$ and $\alpha$ exchanged by $\tilde{\beta}$, $\tilde{a}$,
$\tilde{C}$, $\tilde{\nu}$, $\tilde{D}$ and $\tilde{\alpha}$. By our
assumptions, $A=\tilde{A}$. As a consequence, by Lemma
\ref{lemma:LevyDiffusionGeneratorRewrite} and the uniqueness result of
Lemma \ref{lemma:LevyTypeGeneratorUniqueness}, we find that for all
$x\in\RR^p$ and $i\le p$, we have
\begin{align}
  \beta_i(x)
  &=\tilde{\beta}_i(x),\label{eq:diffDist1}\\
  a(x)Ca(x)^t &= \tilde{a}(x)\tilde{C}\tilde{a}(x)^t,\label{eq:diffDist2}\\
  T^{a(x)}(\nu) &= T^{\tilde{a}(x)}(\tilde{\nu}).\label{eq:diffDist3}
\end{align}
Now let $\Delta>0$ and $t_k = k\Delta$. The Euler scheme for the process
$X$ is a Markov process given by the update function having $i$'th coordinate
\begin{align}
  G(x, U_k)^i &= x^i + \sum_{j=1}^d a_{ij}(x) U_k^j,
\end{align}
with $U_k = Z_{t_k} - Z_{t_{k-1}}$. The Euler scheme for the process
$\tilde{X}$ is likewise given by the update function having $i$'th coordinate
\begin{align}
\tilde{G}(x, V_k)^i &= x^i + \sum_{j=1}^{d'} \tilde{a}_{ij}(x) V_k^j,
\end{align}
with $V_k = \tilde{Z}_{t_k} - \tilde{Z}_{t_{k-1}}$. The characteristic function of $G(x, U_k)$ is
\begin{align}
  E\exp(iu^t (x + a(x)U_k))
  &= \exp(iu^tx)E\exp(i u^t a(x)(Z_{t_k}-Z_{t_{k-1}}))\notag\\
  &= \exp(iu^tx)E\exp(i (a(x)^tu)^t (Z_{\Delta}-Z_0)),
\end{align}
for $u\in\RR^p$. By (\ref{eq:charfun}) and some algebraic manipulations, we therefore have
\begin{align}
  \log E\exp(iu^t (x + a(x)U_k))&=iu^t (x + a(x) \Delta\alpha)-\frac{1}{2}\Delta u^ta(x) C a(x)^t u\\
  &-\Delta\int_{\RR^d}e^{i u^ta(x)y}-1-i u^ta(x)y 1_D(y)\dv \nu(y) \notag\\
  &=iu^t(x+\Delta\beta(x)) -\frac{1}{2}\Delta u^ta(x) C a(x)^tu\notag\\
  &-\Delta\int_{\RR^d}e^{i u^t y} - 1 - iu^t y 1_{E}(y) \dv T^{a(x)}(\nu)(y)\notag.
\end{align}
Making the same calculations for the characteristic function of
$\tilde{G}(x,V_k)$ and applying (\ref{eq:diffDist1}), (\ref{eq:diffDist2}) and
(\ref{eq:diffDist3}), we conclude that (\ref{eq:distid}) holds for the
two Euler scheme update functions. Lemma \ref{lem:markovintervention}
therefore allows us to conclude that the postintervention Euler SEMs have
the same distributions. Using Lemma \ref{lemma:EulerConvergence}, we
conclude that the postintervention distributions obtained from the two
SDEs are equal. 
\hfill$\Box$


\textbf{Acknowledgements.} The authors would like to thank Marloes Maathuis
for fruitful discussions and comments. We also thank two anonymous
reviewers and an associate editor whose comments led to considerable improvements of the
paper. Finally we thank Steffen Lauritzen, who suggested that a result
similar to Lemma \ref{lem:markovintervention} could be used to prove
Theorem \ref{theorem:LevyDiffusionIdentifiability} via the Euler
scheme. 

\bibliographystyle{amsplain}

\bibliography{full}

\end{document}